\documentclass[a4paper, 12pt]{scrartcl}

\usepackage[utf8]{inputenc}

\usepackage{hyperref}
\usepackage{amsmath}  
\usepackage{amsthm}   
\usepackage{amssymb}  
\usepackage{xspace}   
\usepackage{tikz-cd}
\usepackage{xifthen}  
\usepackage{stmaryrd} 
\usepackage{calc}     
\usepackage[sc,osf]{mathpazo} 
\usepackage{upgreek}  
\usepackage{amsopn}   
\usepackage{makeidx}  
\usepackage[euler-digits,small]{eulervm} 
\usepackage{mathrsfs}
\usepackage[a4paper]{geometry}
\usepackage[mathcal]{euscript}
\usepackage[matha, mathb]{mathabx}  
\usepackage{empheq}
\usepackage{rotating}		

\SetSymbolFont{stmry}{bold}{U}{stmry}{m}{n}

\usetikzlibrary{decorations.markings}

\theoremstyle{plain}
\newtheorem{theo}{Theorem}[section]
\newtheorem{lem}[theo]{Lemma}

\newtheorem{prop}[theo]{Proposition}

\theoremstyle{definition}
\newtheorem{defi}[theo]{Definition}

\newtheorem{nota}[theo]{Notation}

\theoremstyle{remark}
\newtheorem{rem}[theo]{Remark}
\newtheorem{exa}[theo]{Example}

\usetikzlibrary{matrix,arrows,decorations.pathmorphing}

\DeclareMathOperator{\de}{d}
\DeclareMathOperator{\SF}{\operatorname{II}}		

\newcommand*{\N}{\ensuremath{\mathbb{N}}\xspace}
\newcommand*{\R}{\ensuremath{\mathbb{R}}\xspace}
\newcommand*{\Fr}[2]{\ensuremath{\operatorname{Fr}_{\mathbf{#1}}(#2)}\xspace}
\newcommand*{\Ad}{\ensuremath{\operatorname{Ad}}\xspace}
\newcommand*{\Sp}{\ensuremath{\mathbf{Sp}}\xspace}
\newcommand*{\id}{\ensuremath{\operatorname{id}}\xspace}
\newcommand*{\Qu}{\ensuremath{\mathbb{H}}\xspace}
\newcommand*{\h}{\ensuremath{\mathfrak{h}}\xspace}
\newcommand*{\g}{\ensuremath{\mathfrak{g}}\xspace}
\newcommand*{\inv}[1]{\ensuremath{{#1}^{-1}}\xspace}
\newcommand*{\Gl}{\ensuremath{\mathbf{Gl}}\xspace}
\newcommand*{\im}{\ensuremath{\operatorname{im}}\xspace}
\newcommand*{\Rea}{\ensuremath{\operatorname{Re}}\xspace}
\newcommand*{\pr}{\ensuremath{\operatorname{pr}}\xspace}
\newcommand*{\SO}{\ensuremath{\mathbf{SO}}\xspace}

\newcommand*{\LieAlg}[2][empty]{
\ifthenelse{\equal{#1}{empty}}
{\ensuremath{\mathfrak{#2}}}
{\ensuremath{\mathfrak{#2}(#1)}}
}

\makeindex

\begin{document}


\title{Some Remarks on the Hyperkähler Reduction}
\author{Robin Raymond\thanks{This work has been funded by the RTG 1493. I'd like
        to thank my supervisor Victor Pidstrygach.}}

\pagenumbering{arabic}
\maketitle
\begin{abstract}
    We consider a \emph{hyperkähler reduction} and describe it via frame
    bundles. Tracing the connection through the various reductions, we recover
    the results of \cite{gocho1992}. In addition, we show that the fibers of
    such a reduction are necessarily totally geodesic. As an independent result,
    we describe O'Neill's submersion tensors \cite{oneill1966} on principal
    bundles.
\end{abstract}

\tableofcontents

\section {Introduction}

The Hyperkähler Reduction is a cousin of the Symplectic Reduction applicable
to the setting where the starting manifold $M$ is hyperkähler and the involved
data, the action of an auxiliary group $G$ and the moment map $\mu$, respect
this structure. It is well known, that this implies that the \emph{final
    manifold}, the quotient of a preimage of a central regular value of $\mu$ by $G$,
also is a hyperkähler manifold. This however is not all that is special about
the hyperkähler reduction.

In their paper \cite{gocho1992} T. Gocho and H. Nakajima find some interesting
relations between various geometrical quantities involved in this construction.
The paper uses various calculations in the tangent bundle to show these
relations.

We will present a different approach in this work by \emph{lifting} the
calculation onto the involved principal bundles. Although quite a bit longer
than the original work, it highlights the role the quaternionic structure plays
in the construction. The length can be partly attributed to the need to
introduce basic notions in this setting, e.g. the section \emph{\ref{rmsub}
Riemannian Submersions} which recovers the fundamentals of O'Neill's theory in
the principal bundle setting.

The aim of this paper is to show that these relations can be derived
fundamentally from the structure of quaternionic matrices, when embedded into
real matrices. It does so, by first deriving equation \eqref{realfinaleq}, which
does not need the involved quaternionic structures. Then this equation is
compared to the \emph{quaternionic world} \eqref{finalequ}, and this comparison
yields all the relations that we long for. It then just remains to decipher the
implied relations for the quaternionic components.

The section \emph{\ref{sec-defi} Definitions} recalls the basic notions involved
in hyperkähler geometry and in particular in a hyperkähler reduction. Of utmost
importance to the next sections are the notions of reduction and extension of
principal bundles. Further it describes a recipe to compare forms on the
manifolds and the involved principal bundles.

Section \emph{\ref{sec-setting} Setting} first discusses the tangent bundle of
$M$ and how its quaternionic structure behaves with respect to the reduction.
This structure allows for various reductions of the principal bundle of frames
of $M$. These bundles lie at the heart of the construction in this work.

The following section inspects the involved forms with respect to the bundles
discussed. Concretely we will trace the reductions of the Levi-Civita connection
and tautological form starting from the principal bundle of frames of $M$ all
the way to the principal bundle of frames of the quotient $N$. A quick excursion
is made in this section, explaining the fundamentals of Riemannian Submersions
in the principal bundle language.

The  last section \emph{\ref{sec-final} Final Results} uses the preceding work to recover the
results of Gocho and Nakajima, and show a small novelty. It is this section
where the relation between the quaternionic structure and the results is
investigated.

I'd like to thank my supervisor Victor Pidstrygach for the idea of this project
and the countless times he assisted me. I'd also like to thank Florian Beck for
proofreading a draft of this work.

\section{Definitions} \label{sec-defi}
Let us define some standard notions. Throughout this paper, let $M$ be a smooth
oriented Riemannian manifold of dimension $4m\in \N$, and $G$ a smooth Lie group
of dimension $k\in \N$.

\begin{nota}
    By $\Fr{SO}{M}$ we denote the \emph{principal bundle of orthonormal frames} on $M$,
    \begin{align*}
        \Fr{SO}{M} = \bigl\{ p\colon \R^{4m} \to T_{x}M : \text{$p$ is an oriented
            orthogonal isomorphism} \bigr\}.
    \end{align*}
\end{nota}

\begin{nota}
    By $\theta^M \in \Omega^1(\Fr{SO}{M}, \R^{4m})$ we denote the \emph{soldering form} of $\Fr{SO}{M}$
    \begin{align*}
        \theta^M_p(\xi) = p^{-1}\circ D\pi_p(\xi), \qquad p\in \Fr{SO}{M},\,\, \xi\in T_p\Fr{SO}{M},
    \end{align*}
    where $\pi\colon\Fr{SO}{M} \to M$ is the projection.
\end{nota}

Let $\varphi\in \Omega^1{(\Fr{SO}{M}, \LieAlg{so}(4m))}^{\SO(4m)}$ denote the
Levi-Civita
connection of $(M,g)$. Then $\varphi$ satisfies
\begin{itemize}
    \item $\text{R}_g^*\varphi = \Ad_{g^{-1}} \circ \varphi$, for all $g\in \SO(4m),$
    \item $\varphi(K^{\xi}) = \xi$ for all $\xi \in \mathfrak{so}(4m)$, where $K^{\xi}$ is the fundamental vector field
        to the lie algebra element $\xi$, i.e. $$K^{\xi}_p = \left.\frac{\de}{\de t}\right|_{t=0}(p\exp(t\xi)),$$
    \item $\text{d}\theta + \varphi \wedge \theta = 0$, i.e. $\varphi$ has zero torsion.
\end{itemize}

\begin{defi}[Hyperkähler Manifold]
    A Riemannian manifold $(M,g)$ with a triple of almost complex structures $I,J,K$,
    \begin{align*}
        I,J,K \colon TM \to TM, \qquad I^2=J^2=K^2 = -\id_{TM},
    \end{align*}
    which satisfy the quaternionic relation $IJ = K$ and are compatible with the metric,
    \begin{align*}
        g(-,-) = g(I-,I-) = g(J-, J-) = g(K-,K-),
    \end{align*}
    is called a hyperkähler manifold (hk-manifold) if the two-forms corresponding to $I,J$ and $K$
    are closed, i.e.
    \begin{align*}
        d\omega_A = 0, \qquad \omega_A(-,-) = g(A-,-), \qquad A\in\left\{I,J,K\right\}.
    \end{align*}
\end{defi}

\begin{prop}[Alternative Characterization]
    $(M^{4m},g)$ is a hyperkähler manifold if and only if the structure group
    of $\Fr{SO}{M}$ reduces to $\Sp(m)$ and the
    Levi-Civita connection on $\Fr{SO}{M}$ reduces to a connection on
    \begin{align*}
        \Fr{Sp}{M} = \bigl\{ p\colon \Qu^{m} \to T_{x}M : \text{$p$ is a \Qu-linear isomorphism} \bigr\},
    \end{align*}
    i.e.\ the horizontal subspaces are
    tangent to the submanifold $\Fr{Sp}{M} \subset \Fr{SO}{M}$.
\end{prop}

Note that in the dual formulation the condition on the horizontal subspaces is that $\varphi$ reduces to a connection
on $\Fr{Sp}{M}$. Precisely this means that $\lambda_*j^*\varphi$ is a connection on $\Fr{Sp}{M}$,
where $j\colon \Fr{Sp}{M} \to \Fr{SO}{M}$ and $\lambda\colon \Sp(m) \to \SO(4m)$ are the inclusions
and $\lambda_*\colon \mathfrak{sp}(m)\to \mathfrak{so}(4m)$ is the derivative of $\lambda$.

\begin{defi}[Hyperkähler Action]
    We say a group $G$ acts hyperkähler on a hyperkähler manifold $(M,g,I,J,K)$, if $G$ acts on $M$ and
    this action preserves the metric $g$ and the hyperkähler structures $I,J$ and $K$, i.e.
    \begin{align}
        R_h^* \omega_A = \omega_A \quad \forall A\in\left\{I,J,K\right\}, \qquad R_h^* g = g,
    \end{align}
    for all $h\in G$. (In this case we used a right action of $G$ on $M$, but this definition does not require so).
\end{defi}

\begin{defi}[tri-hamiltonian action]
    A hyperkähler action of $G$ on $M$ is called a tri-hamiltonian action,
    if $G$-equivariant moment maps
    \begin{align}
        \mu_I,\mu_J, \mu_K\colon M\to \mathfrak{g}^*
    \end{align}
    exist, i.e.
    \begin{align}
        \mu_A(x.h) = \Ad_{h^{-1}}^*\circ \mu_A(x) \qquad \forall x\in M, \quad \forall h \in
        G, \quad \forall A\in\left\{I,J,K\right\},  \label{firsttri} \\
        \left<\xi, \text{d}\mu_A(\eta)\right> = \omega_A(K^{\xi}, \eta) \qquad \forall
        \eta\in TM,\quad \forall \xi \in \mathfrak{g},\quad \forall A\in\left\{I,J,K\right\}.
    \end{align}
\end{defi}
The moment maps of a tri-hamiltonian action are also often considered together as a map
$\mu=(\mu_I,\mu_J,\mu_K)\colon M \to \R^3\otimes \,\,\mathfrak{g}^*$.

\subsection{Reduction and Extensions}

Let $\pi\colon P \to M$ be a principal bundle with structure group $G$. A reduction of $P$ is a principal bundle $Q\to M$
with structure group $H$ and maps
\begin{align}
    \lambda \colon H\to G, \qquad f\colon Q\to P,
\end{align}
a Lie homomorphism and a smooth map respectively, such that the following diagram commutes.

\begin{center}
    \begin{tikzpicture}
        \matrix(m) [matrix of math nodes,row sep=3em,column sep=4em,minimum width=2em]
        {
            Q\times{H} &  & P\times{G}\\
            Q & & P \\
            & M & \\
        };
        \path[-stealth]
        (m-1-1) edge node [above] {$f\times \lambda $} (m-1-3)
        edge (m-2-1)
        (m-1-3) edge (m-2-3)
        (m-2-1) edge node [above] {$f$} (m-2-3)
        (m-2-1) edge (m-3-2)
        (m-2-3) edge (m-3-2);
    \end{tikzpicture}
\end{center}
The vertical maps above are the group actions on the principal bundles.
An extension of $P$ is a principal bundle $\tilde{Q}\to M$ of structure group $\tilde{H}$ with maps $\tilde{\lambda}\colon G\to \tilde{H}$ and
$\tilde{f}\colon P\to \tilde{Q}$, such that $P$ is a reduction of $\tilde{Q}$.

Given a connection $\phi^{P}$ on $P$, then there is a unique connection $\phi^{\tilde{Q}}$ on $\tilde{Q}$ such that
\begin{align}
    \tilde{f}^*\phi^{\tilde{Q}} = \tilde{\lambda}_*\circ \phi^{P},
\end{align}
where $\tilde{\lambda}_*$ is the derivative of $\tilde{\lambda}$ (see
e.g.~\cite[Satz 4.1]{baum}). In this sense, a connection is always extendable. If two connections
satisfy the equation above, we say that $\phi^P$ extends to $\phi^{\tilde{Q}}$ and $\phi^{\tilde{Q}}$ reduces to $\phi^P$.

On $Q$ the situation is somewhat more complicated. We will only discuss the situation for the simplest case where $f=i$ and $\lambda$ are the inclusions.
\begin{prop}[Reduction of a connection]\label{reduction}
    If $\mathfrak{g} = \mathfrak{h}\oplus \mathfrak{f}$ as $H$-representations, i.e.
    $\mathfrak{f} \subset \mathfrak{g}$
    is a vector space complement of $\mathfrak{h} \subset \mathfrak{g}$,
    with the property that
    \begin{align} \label{reductionCond}
        \Ad_H(\mathfrak{f}) \subset \mathfrak{f},
    \end{align}
    then $\pr_{\mathfrak{h}} \circ i^*\phi^P$ is a connection on $Q$, where the projection is with respect to the
    decomposition given above.
\end{prop}
\begin{proof}
    The only thing to note is, that the condition $\Ad_H(\mathfrak{f}) \subset
    \mathfrak{f}$ (together with $\Ad_H(\h) \subset \h$) implies that
    $\pr_{\mathfrak{h}}$ commutes with $\Ad_h$ for all $h\in H$. The necessary conditions are then easily checked.
\end{proof}

\begin{defi}
    We say that $\varphi$ \emph{reduces} to $Q$ when the
    horizontal subspaces are tangent to the subbundle $Q \subset P$. In the dual
    formulation this is true if and only if the pulled back connection takes
    values in the Lie algebra $\h$, so that no projection is necessary.
\end{defi}

Note that a projected connection as in the lemma above can be extended back to $P$. This will however yield a different
connection if the original one was not reducible. This also implies that there are in general multiple connections
on $P$ that project onto a given connection on $Q$.

\begin{rem} \label{naturalsoldering}
    Let $\iota \colon Q \to \Fr{SO}{M}$ denote a reduction of the frame bundle.
    We call the pull back $\theta^Q$ of $\theta^M$ to $Q$ again \emph{soldering
        form} of $Q$. Since the diagram
    \begin{center}
        \begin{tikzpicture}
            \matrix(m) [matrix of math nodes,row sep=3em,column sep=4em,minimum width=2em]
            {
                Q & & \Fr{SO}{M} \\
                & M & \\
            };
            \path[-stealth]
            (m-1-1) edge node [left,below] {$\pi^Q$} (m-2-2)
            (m-1-1) edge node [above] {$\iota$} (m-1-3)
            (m-1-3) edge node [right,below] {$\pi^M$} (m-2-2);
        \end{tikzpicture}
    \end{center}
    commutes, we have that for all $p\in Q$ and $\xi\in T_p Q$
    \begin{align}
        \theta^Q_p(\xi) &= (\iota^* \theta^M)_p (\xi) = \theta^M_{\iota(p)}(\xi)
        = \inv{\iota(p)} \circ (D\pi^M)_{\iota(p)} \circ D\iota_p(\xi) \\
        &= \inv{\iota(p)} \circ (D(\pi^M\circ \iota))_p (\xi) = \inv{\iota(p)} \circ
        (D\pi^Q)_p(\xi),
    \end{align}
    so that $\theta^Q_p = \inv{\iota(p)} \circ D\pi^Q_p$.
    In this sense the construction is natural.
\end{rem}

\subsection{The Correspondence of Forms}
Having a principal bundle of frames $\Fr{Gl}{M}$ (or any reduction of it) over a manifold $M$ induces a correspondence between certain forms
on the base manifold and the bundle. We will use this correspondence to compare our approach and the one taken in
\cite{gocho1992}.

\begin{lem}[Correspondence of forms]\label{corr}
    There is a one-to-one correspondence between horizontal, equivariant and
    $\mathfrak{gl}(4m)$-valued one-forms on the principal bundle of frames,
    and (global) sections of the vector bundle $T^*M\otimes \text{End}(TM)$.
\end{lem}
\begin{rem}
    Note that this is a special case of the correspondence between representation
    valued forms on a principal bundle and forms with values in associated vector
    bundles on the base. In the presence of the soldering form, we can give a
    simple explicit description.
\end{rem}

\begin{proof}
    Let $\omega$ be a horizontal and equivariant one-form on the principal bundle. We induce the wanted section as follows. If $x\in M$ and $\xi, \eta\in T_{x}M$,
    let $p\in \Fr{Gl}{M}$ be any frame in the fiber of $\pi$ over $x$. Define
    \begin{align}
        s(\omega)(\xi,\eta) = p\omega(\bar{\xi})\theta(\bar{\eta}),
    \end{align}
    where $\theta$ is the solder form of $\Fr{Gl}{M}$ and $\bar{\xi}$ and
    $\bar{\eta}$ are lifts of $\xi$ and $\eta$ to $p\in \Fr{Gl}{M}$, i.e. $D\pi(\bar{\xi}) = \xi$ and
    $D\pi(\bar{\eta}) = \eta$. This is well defined, because for a different choice of lifts $\tilde{\xi}$ and $\tilde{\eta}$, the differences $\Delta\xi = \tilde{\xi} - \bar{\xi}$
    and $\Delta\eta = \tilde{\eta} - \bar{\eta}$ are vertical, but $\omega$ and
    $\theta$ are both horizontal forms. A different choice of frame $q = p.g\in
    \Fr{Gl}{M}$, leads to the calculation
    \begin{align}
        q\omega(\bar{\xi})\theta(\bar{\eta}) &= q\omega(\bar{\xi})q^{-1}(\eta) = p.g \omega(\bar{\xi}){(p.g)}^{-1}(\eta) = pg \omega(\bar{\xi}) g^{-1}p^{-1}(\eta) \\
        &= p \Ad_g(\omega(\bar{\xi})) \theta(\bar{\eta}) = p\omega(DR_{g^{-1}}\bar{\xi}) \theta(\bar{\eta}) = p\omega(\tilde{\xi}) \theta(\bar{\eta}) \nonumber \\
        &= p\omega(\bar{\xi})\theta(\bar{\eta}), \nonumber
    \end{align}
    where we have used the equivariance of $\omega$, $R_g^*\omega = \Ad_{g^{-1}}\omega$, and the fact that $DR_g$ maps lifts into lifts, since $\pi \circ R_g = \pi$ and
    therefore $D\pi \circ DR_g = D\pi$ for all $g\in \Gl(m)$. By abuse of notation
    $\overline{\eta}$ denotes a lift to both $q$ and $p$ in $T\Fr{Gl}{m}$.

    Note that we have only needed $\Gl(m)$ for the fact that $\Ad_g(\xi) = g\xi g^{-1}$, so this will be true for all principal bundles in this work, if we adjust
    the vector bundle in which the sections are taken.

    The inverse map, sending a section to a form on the principal bundle is defined by
    \begin{align}
        \omega(s)(\xi) = p^{-1}s(D\pi(\xi))p,
    \end{align}
    where $p\in \Fr{Gl}{M}$ is some frame, $\xi\in T_{p}\Fr{Gl}{M}$ and $s\in \Gamma(T^*M\otimes \text{End}(TM))$ is the section. This form is clearly a horizontal
    $\mathfrak{gl}(m)$-valued one-form. It is also equivariant because
    \begin{align}
        R_g^* \omega(s)(\xi) &= {(pg)}^{-1} s(D\pi\circ DR_g (\xi)) pg = g^{-1}p^{-1} s(D\pi(\xi)) pg \\
        &= \Ad_{g^{-1}} \omega(s)(\xi). \nonumber
    \end{align}

    It is easy to show that these two maps are inverse of each other, which concludes the proof.
\end{proof}

\begin{defi}[Corresponding forms]
    As denoted in the proof above, the section of $T^*M\otimes \text{End}(TM)$ corresponding to $\omega$ is denoted by $s(\omega)$, and the form corresponding
    to a section $s$ by $\omega(s)$.
\end{defi}

Note that this result remains true for reductions of the basis bundle, if we adjust the vector bundle in which the sections are taken. For example, the
above mentioned forms on $\Fr{SO}{M}$ correspond to sections in $T^*M \otimes
\LieAlg{so}(TM)$ and the forms on $\Fr{Sp}{M}$ to sections of $T^*M \otimes
\LieAlg{sp}(TM)$.

\begin{exa}[Difference form]
    A well known example of this correspondence is between the difference form of two connections on a principal bundle, and the difference tensor of the two
    associated covariant derivatives. This follows immediately from equation (\ref{oneillcon}).
\end{exa}

\section{Setting}\label{sec-setting}

We will recover the results from~\cite{gocho1992} for principal bundles.

Let $(M,g)$ be an Riemannian manifold of dimension $4m\in \N$, and let $M\curvearrowleft G$ be a tri-hamiltonian action of $G$ on $M$.
Let $k\in \N$ be the dimension of the Lie group $G$.
We denote the momentum map by $\mu\colon M \to \R^3 \otimes \,\,\mathfrak{g}^*$. We assume that
$0\in \R^3\otimes \,\,\mathfrak{g}^*$ is a regular value of $\mu$. This implies
that $G$ acts on the submanifold $\mu^{-1}(0)$, because equation
(\ref{firsttri}) guarantees that for $x\in \mu^{-1}(0)$, i.e.
$\mu_A(x) = 0$ for all $A$, we have
\begin{align}
    \mu_A(x.h) = \Ad_h^*\circ \mu_A(x) = 0, \qquad \forall h \in G,
\end{align}
and hence $x.h\in \mu^{-1}(0)$.

We assume further that this action is free and proper, so that the quotient $\mu^{-1}(0)/G$ is a Hausdorff space, and define $N:= \mu^{-1}(0)/G$.

\begin{center}
    \begin{tikzpicture}
        \matrix (m) [matrix of math nodes,row sep=3em,column sep=4em,minimum width=2em]
        {
            \mu^{-1}(0) & M \\
            N &  \\
        };
        \path[-stealth]
        (m-1-1) edge node [right] {$\pi$} (m-2-1)
        edge node [above] {$\iota$} (m-1-2);
    \end{tikzpicture}
\end{center}

We will show that $N$ also is a hyperkähler manifold, and that the second fundamental form of $\mu^{-1}(0)$ in $M$ is given
by the Hessian of $\mu$, compare~\cite{gocho1992} and~\cite{hitchin1987}.

\subsection{The Splitting of \texorpdfstring{$TM$}{TM}}

The tri-hamiltonian action $M \curvearrowleft G$ splits the vector bundle $TM$ over $\mu^{-1}(0)$, i.e.\ the ambient bundle
\begin{align}
    \iota^*TM,
\end{align}
in the following way.

\begin{prop}
    If $x\in \mu^{-1}(0)$, we have
    \begin{align}
        T_{x}M = T_x\mu^{-1}(0) \oplus T_x\mu^{-1}{(0)}^{\perp} = H_x \oplus \mathfrak{g} \oplus T_x\mu^{-1}{(0)}^{\perp},
    \end{align}
    where $\mathfrak{g} \subset T_xM$ is defined by the fundamental vector fields,
    i.e. the image of $K\colon \LieAlg{g} \to \Gamma(TM)$, and $H_x$ is the orthogonal complement to
    $\mathfrak{g}$ in $T\mu^{-1}(0)$ with respect to the metric $g$. All direct sums
    are orthogonal.

    Then $H_x$ is a quaternionic subspace of $T_xM$ and
    \begin{equation}
        T_x\mu^{-1}(0)^{\perp} = I \g\oplus J\g \oplus K\g.
    \end{equation}
\end{prop}

\begin{proof}

    If $\xi \in \mathfrak{g}$ and $\eta\in T_x\mu^{-1}(0)$ then $\eta$ is
    tangent to a level set of $\mu$, i.e.  $\text{d}\mu(\eta) = 0$, which
    implies for $A\in\left\{I,J,K\right\}$
    \begin{align}
        g(AK^{\xi}, \eta) = \omega_A(K^{\xi}, \eta) = \left<\xi,\text{d}\mu_A(\eta)\right> = 0,
    \end{align}
    hence $AK^{\xi} \in T\mu^{-1}{(0)}^{\perp}$ for all $A$.

    Furthermore the sets $I\g, J\g$ and $K\g$ have a trivial intersection.
    Indeed, assume $\xi, \eta\in\g$ with $I\xi = J\eta$. Then $K\xi = \eta$ but
    since $K\xi$ is in $T\mu^{-1}(0)^{\perp}$, $\eta = \xi = 0$.

    Since the codimension of $\mu^{-1}(0)$ in $M$ is $3k$, where $k = \dim G = \dim \mathfrak{g}$, we see that
    \begin{align}
        T_x\mu^{-1}{(0)}^{\perp} = I\mathfrak{g} \oplus J\mathfrak{g} \oplus K\mathfrak{g}.
    \end{align}

    Finally, $I,J$ and $K$ let the orthogonal complement of $H_x$ invariant and are orthogonal, so they also let $H_x$ invariant.

\end{proof}
We conclude that $TM$ splits over $\mu^{-1}(0)$ into two quaternionic sub-bundles
\begin{align} \label{splitting}
    TM = H \oplus \mathfrak{g}\otimes_{\R}\Qu.
\end{align}
Notice that while the first bundle has a quaternionic structure, the second one
has a quaternionic and a real structure. This will become important later on.

The metric $g$ of $M$ induces a metric on $H$. Since $M\curvearrowleft G$ is
hyperkähler and $g$ is $G$-invariant it furnishes $N$ with a Riemannian
metric. Similarly the quaternionic structure on $M$ induces one on $H$ (because
of the quaternionic decomposition above), which in turn induces one on $N$
compatible with the metric. This reduces the principal bundle of orthogonal
frames on $N$ to the structure group $\Sp(n)$ ($n = m - k$, $4n$ is the
dimension of $N$). We will show later that the connection of $N$ reduces so that
$N$ is indeed a hyperkähler manifold.

\subsection{The Principal Bundles}

Similar to the vector bundle $TM$, we may depict the splitting in the principal
bundle setting. Fix a splitting
\begin{equation}\label{splitbundles}
    \R^{4m} = \R^{4n} \oplus \R^{k} \oplus \R^{3k} = \Qu^n \oplus \Qu^k
\end{equation}

Now we can ask frames $p\colon \R^{4m} \to T_xM$ to respect various degrees of
the structure. Let $x\in\mu^{-1}(0)$.

\begin{itemize}
    \item $p\colon \R^{4m} \to T_xM$ with no condition at all. These frames are in the
        pull back of the frame bundle $\Fr{SO}{M}$ to $\mu^{-1}(0)$, denoted by
        $\iota^{*}\Fr{SO}{M}$.
    \item $p\colon \R^{4m} \to T_xM$ with $p(\R^{4n}\oplus \R^{k}) = T\mu^{-1}(0)$, frames adapted to the
        submanifold $\mu^{-1}(0) \subset M$. This is a principal bundle whose
        structure group is $\SO(4n+k) \times \SO(3k)$, corresponding to the
        possible rotations of the frame in $T\mu^{-1}(0)$ and
        $T\mu^{-1}{(0)}^{\perp}$. We denote it by

        \begin{align}
            \Fr{SO}{\mu^{-1}(0),M} = \left\{p\in \Fr{SO}{M} :\quad
                \im(p|_{\R^{4n+k}}) = T\mu^{-1}(0)\right\},
        \end{align}

    \item $p\colon \R^{4n} \oplus \R^k \to T_x\mu^{-1}(0)$. These frames can be
        identified with frames of $\inv{\mu}(0)$. We denote them with
        $\Fr{SO}{\mu^{-1}(0)}$.
    \item $p\colon \R^{4n} \oplus \R^k \to T_x\mu^{-1}(0)$ with $p(\R^{4n}) =
        H_x$. These frames are frames of $\inv{\mu}(0)$ adapted to the fibration
        $\pi\colon \inv{\mu}(0) \to N$. The principal bundle of these have
        structure group $\SO(4n) \times \SO(k)$ corresponding to the rotations
        in the fiber and its orthogonal complement. We denote the bundle by

        \begin{align}
            \Fr{SO}{N,\mu^{-1}(0)} = \left\{p\in \Fr{SO}{\mu^{-1}(0)} :\quad
                \im(p|_{\R^{4n}}) = H_x \right\},
        \end{align}

    \item $p\colon \R^{4n} \to H_x$. The principal bundle of these frames can
        be identified with the pull back of $\Fr{SO}{N}$ to $\inv{\mu}(0)$ (note
        that we know already that $N$ is a Riemannian manifold). We
        denote it by $\pi^*\Fr{SO}{N}$.
\end{itemize}

We may restrict the principal bundles above to quaternionic frames where it
makes sense. Fix
\begin{equation}
    \Qu^m = \Qu^n\oplus \Qu^k
\end{equation}
respecting \eqref{splitbundles}. This induces the following bundles, where all
frames are $\Qu$-linear.

\begin{itemize}
    \item $p\colon \Qu^m \to T_xM$ are the frames that make up the pull back of
        \begin{align}
            \Fr{Sp}{M} = \bigl\{p \in \Fr{SO}{M} : \text{$p$ is $\Qu$-linear}\bigr\}
        \end{align}
        to $\mu^{-1}(0)$. It is naturally a reduction of $\iota^{*}\Fr{SO}{M}$
        to quaternionic frames, has structure group $\Sp(m)$ and will be denoted
        by $\iota^*\Fr{Sp}{M}$.

    \item $p\colon \Qu^m \to T_xM$ with $p(\Qu^n) = H_x$ and $p(\Qu^k) =
        \g\otimes \Qu$ respecting both the quaternionic and real structure. We
        denote this principal bundle with structure group $\Sp(n)\times \SO(k)$
        by

        \begin{align}
            \Fr{Sp}{N,M} = \left\{p \in \iota^*\Fr{Sp}{M} :\,\,
                \text{im}(p|_{\Qu^n}) = H_x, \,\,\text{im}(p|_{\Qu^{k}}) =
                \mathfrak{g}\otimes\Qu\right\}
        \end{align}
        The frames are adapted to the quaternionic splitting of $T_xM = H_x
        \oplus \g\otimes\Qu$ and respect the real structure of the second,
        $p(\Rea(\Qu^k)) = \Rea(\g\otimes \Qu) = \g$, so in particular (because
        $I,J,K$ are orthogonal) respect the splitting $T\inv{\mu}(0) \oplus
        T\inv{\mu}(0)^{\perp}$.
    \item $p\colon \Qu^n \to H_x$ are the frames of the pulled back bundle
        $\Fr{Sp}{N}$ to $\inv{\mu}(0)$ and is denoted by $\pi^*\Fr{Sp}{N}$.
\end{itemize}

There are plenty of natural maps between these bundles. We will be using the
following.

\begin{itemize}
    \item \emph{Reductions to quaternionic frames, denoted by i:} Some of the real frame
        bundles can be reduced to quaternionic frames, which induces maps from
        the quaternionic world to the real world. This is
        obviously the case for $\Fr{Sp}{M} \to \Fr{SO}{M}$, $\Fr{Sp}{N} \to
        \Fr{SO}{N}$ and their pull backs to $\inv{\mu}(0)$. Finally this is also
        the case for $\Fr{Sp}{N,M} \to \Fr{SO}{\inv{\mu}(0), M}$, because a
        quaternionic frame that respects the splitting $T_x\inv{\mu}(0) \oplus
        T_x\inv{\mu}(0)^{\perp}$, automatically respects the quaternionic
        splitting $H_x \oplus \g\otimes \Qu$, as
        can be seen by applying one of the complex structures to
        $T_x\inv{\mu}(0)^{\perp}$. In
        other words, $(\SO(4n+k)\times \SO(3k))\cap \Sp(m) \cong \Sp(n) \times
        \SO(k)$.
    \item \emph{Reduction to more structured frames, denoted by j:} Some of
        the bundles are simply restrictions of other bundles to frames
        respecting more structures. This is the case for
        \begin{equation}
            \Fr{SO}{\mu^{-1}(0), M} \to \iota^*\Fr{SO}{M}, \qquad
            \Fr{SO}{N,\mu^{-1}(0)} \to \Fr{SO}{\mu^{-1}(0)}
        \end{equation}
        and
        \begin{equation}
            \Fr{Sp}{N,M} \to \iota^*\Fr{Sp}{M}.
        \end{equation}
    \item \emph{Induced maps by pull backs, also denoted by j:} There are of
        course canonical maps $\iota^*\Fr{SO}{M} \to \Fr{SO}{M}$ and similar for
        $\pi\colon \inv{\mu}(0) \to N$ and the quaternionic bundles.

    \item \emph{Restrictions of frames, denoted by k:} Some bundles allow
        natural projections to other bundles by restricting the frame to a
        subspace of its domain. This is the case for
        \begin{equation}
            \Fr{SO}{\inv{\mu}(0),M} \to \Fr{SO}{\inv{\mu}(0)}, \quad
            \Fr{SO}{N,\inv{\mu}(0)} \to \pi^*\Fr{SO}{N}
        \end{equation}
        and
        \begin{equation}
            \Fr{Sp}{N,M} \to \pi^*\Fr{Sp}{N}, \qquad \Fr{Sp}{N,M} \to
            \Fr{SO}{N,\mu^{-1}(0)}.
        \end{equation}
\end{itemize}

The aforementioned bundles are depicted in the following diagram.

\begin{center}
    \begin{sideways}
        \begin{tikzpicture}[
            back line/.style={densely dotted},
            cross line/.style={preaction={draw=white, -, line width=6pt}},
            descr/.style={fill=white,inner sep=2.5pt},scale=0.8, every
            node/.style={scale=0.72}]
            \matrix (m) [matrix of math nodes,row sep=3em,column sep=2em, text height=1.5ex, text depth=0.25ex]
            {
                & \pi^*\Fr{SO}{N} & \Fr{SO}{N,\mu^{-1}(0)} & \Fr{SO}{\mu^{-1}(0)} & &
                \Fr{SO}{\mu^{-1}(0), M} & \iota^*\Fr{SO}{M} &  & \Fr{SO}{M}\\
                \pi^*\Fr{Sp}{N} & & & & \Fr{Sp}{N,M} & \iota^*\Fr{Sp}{M} & & \Fr{Sp}{M} &\\
                & \Fr{SO}{N} & & & & & & & \\
                \Fr{Sp}{N} & & & & & & & &\\
                & & & & \mu^{-1}(0) & & M & & \\
                & & & & N & & & & \\
            };
            \path[->]
            (m-5-5) edge node [right] {$\pi$} (m-6-5)
            edge node [above] {$\iota$} (m-5-7)
            (m-2-8) edge node [right] {$\pi^{\Qu}_M$} (m-5-7)
            (m-1-9) edge [bend left = 30] node [right] {$\pi^{\R}_M$} (m-5-7)
            (m-1-7) edge [back line, bend left = 30] node [right] {$\hat{\pi}^{\R}$} (m-5-5)
            (m-1-6) edge [back line, bend right = 10] node [descr] {$\tilde{\pi}^{\R}$} (m-5-5)
            (m-2-5) edge node [descr] {$\tilde{\pi}^{\Qu}$} (m-5-5)
            (m-2-6) edge [bend left = 5] node [right] {$\hat{\pi}^{\Qu}$} (m-5-5)
            (m-3-2) edge [back line] node [below] {$\pi^{\R}_N$} (m-6-5)
            (m-1-2) edge [back line] node [right] {$\bar{\pi}^{\R}$} (m-5-5)
            (m-4-1) edge node [above] {$\pi^{\Qu}_N$} (m-6-5)
            (m-2-1) edge node [left] {$j_4^{\Qu}$} (m-4-1)
            (m-4-1) edge [back line] node [above, sloped] {$i_5$} (m-3-2)
            (m-1-2) edge [back line] node [left=8pt, above=5pt] {$j_4^{\R}$} (m-3-2)
            (m-2-1) edge [cross line] node [above] {$\bar{\pi}^{\Qu}$} (m-5-5)
            (m-1-6) edge node [above] {$k_1^{\R}$} (m-1-4)
            (m-1-3) edge node [above] {$j_3^{\R}$} (m-1-4)
            (m-1-3) edge node [above] {$k_2^{\R}$} (m-1-2)
            (m-1-4) edge [back line] node [above, right] {$\pi_{\mu}$} (m-5-5)
            (m-1-3) edge [back line] node [above, right] {$\pi^{\prime}$} (m-5-5)
            (m-2-1) edge node [above, sloped] {$i_4$} (m-1-2)
            (m-1-6) edge node [above] {$j_2^{\R}$} (m-1-7)
            (m-1-7) edge node [above] {$j_1^{\R}$} (m-1-9)
            (m-2-5) edge node [above, sloped] {$i_3$} (m-1-6)
            (m-2-6) edge node [above, sloped] {$i_2$} (m-1-7)
            (m-2-8) edge node [above, sloped] {$i_1$} (m-1-9)
            (m-2-5) edge [cross line] node [above] {$k$} (m-1-3)
            (m-2-5) edge [cross line] node [above] {$j_2^{\Qu}$} (m-2-6)
            (m-2-5) edge [cross line] node [right=15pt,above] {$k_1^{\Qu}$} (m-2-1)
            (m-2-6) edge node [above right] {$j_1^{\Qu}$} (m-2-8);
        \end{tikzpicture}
    \end{sideways}
\end{center}

\section{The Induced Connections}
In this chapter we will start with the Levi Civita connection on $\Fr{SO}{M}$ and chase it through the diagram. This will show that
$N$ is indeed a hyperkähler manifold and recover the results from~\cite{gocho1992}.

\subsection{Forms on \texorpdfstring{$\Fr{Sp}{M}$}{Fr(Sp)(M)}}

Starting with the solder form $\theta^{M,\R}$ and the Levi Civita connection $\varphi^{M,\R}$ on $\Fr{SO}{M}$,
we first induce the forms $\theta^{M,\Qu}$ and $\varphi^{M,\Qu}$ on $\Fr{Sp}{M}$, by pulling back with $i_1$,
\begin{align*}
    \theta^{M,\Qu} = i_1^*\theta^{M,\R}, \qquad
    \varphi^{M,\Qu} = i_1^*\varphi^{M,\R}.
\end{align*}
Since $M$ is a hk-manifold, $\varphi^{M,\Qu}$ is a connection on
$\Fr{Sp}{M}$ satisfying the pulled back structure equation
\begin{align*}
    \de \theta^{M,\Qu} + \varphi^{M,\Qu} \wedge \theta^{M,\Qu} = 0.
\end{align*}

As remarked in (\ref{naturalsoldering}) $\theta^{M,\Qu}$ is again the soldering
form of $\Fr{Sp}{M}$, hence $\varphi^{M,\Qu}$ is a torsion free connection on
$\Fr{Sp}{M}$.

\subsection{Forms on \texorpdfstring{$\iota^*\Fr{SO}{M}$}{i*Fr(SO)(M)} and
    \texorpdfstring{$\iota^*\Fr{Sp}{M}$}{i*Fr(Sp)(M)}}

The solder forms and connection forms on $\Fr{SO}{M}$ and $\Fr{Sp}{M}$ further
induce connections on the ambient principal bundles $\iota^*\Fr{Sp}{M}$ and
$\iota^*\Fr{SO}{M}$ which we will denote by
$\hat{\varphi}^{\R},\hat{\theta}^{\R}$ and
$\hat{\varphi}^{\Qu},\hat{\theta}^{\Qu}$ with the obvious choice. The
$\hat{\varphi}$ are connections, since we do not change the fibers of the
principal bundle (although some may be discarded).  It is also a torsion free
connection, since the structural equation $\de\theta + \varphi \wedge \theta =
0$ survives the pull back and by
using remark (\ref{naturalsoldering}), the pulled back solder forms are natural

\begin{align*}
    \hat{\theta}^{\R}_p(\xi) = p^{-1}\circ D\hat{\pi}^{\R}_p(\xi), \qquad
    \hat{\theta}^{\Qu}_q(\eta) = q^{-1}\circ D\hat{\pi}^{\Qu}_q(\eta),
\end{align*}
where $p\in \iota^*\Fr{SO}{M}$, $\xi \in T_p\iota^*\Fr{SO}{M}$ and $q\in \iota^*\Fr{Sp}{M}$, $\eta \in T_q\iota^*\Fr{Sp}{M}$.

\subsection{Forms on \texorpdfstring{$\Fr{SO}{\mu^{-1}(0),
            M}$}{Fr(SO)(u(-1)(0),M)}}\label{splittingsec}

The next step is to transfer these forms to the principal bundle
\begin{align*}
    \Fr{SO}{\mu^{-1}(0), M} = \left\{p\in \iota^*\Fr{SO}{M}: \text{im}(p|_{\R^{4n+k}}) = T\mu^{-1}(0)\right\},
\end{align*}
which has structure group $\SO(4n+k) \times \SO(3k)$.

Different to before is that $\Fr{SO}{\mu^{-1}(0), M}$ is in general not horizontal in the ambient bundle, hence we
need to project in order to get a connection.

Lemma (\ref{reduction}) allows us to define connections on the adapted
frame bundles $\Fr{SO}{\mu^{-1}(0),M}$ and $\Fr{Sp}{N,M}$. With the inclusion
\begin{align*}
    i\colon \SO(4n+k)\times \SO(3k) \to \SO(4m), \qquad (A,B) \mapsto \left(\begin{array}{cc}
            A & 0 \\ 0 & B
        \end{array}\right),
\end{align*}
we get the Lie algebra decomposition (as vector spaces)
\begin{align*}
    \mathfrak{so}(4m) = \mathfrak{so}(4n+k) \oplus \mathfrak{so}(3k) \oplus \mathfrak{f},
\end{align*}
where
\begin{align*}
    \mathfrak{f} = \left\{\left(\begin{array}{cc}
                0 & C \\ -C^t & 0
            \end{array}\right) \in \mathfrak{so}(4m) : C\in \text{Mat}(4n+k,3k)\right\}.
\end{align*}
If $A\in \text{im}(i)$ and $\xi \in \mathfrak{f}$, then $\Ad_A(\xi) = A\xi A^{-1} \in \mathfrak{f}$, hence we have a
connection $\tilde{\varphi}^{\R} = \pr_{\mathfrak{so}(4n+k) \oplus \mathfrak{so}(3k)} \circ j_2^{\R *}\hat{\varphi}^{\R}$
on $\Fr{SO}{\mu^{-1}(0), M}$. This connection naturally decomposes into two equivariant one-forms $\phi_1^{\R}$ and $\phi_2^{\R}$ with values in
$\mathfrak{so}(4n+k)$ and $\mathfrak{so}(3k)$ respectively.

We can go ahead and extend $\tilde{\varphi}^{\R}$ back to $\iota^*\Fr{SO}{M}$, which gives us a connection $\hat{\varphi}^{\prime \R}$. The difference
form
\begin{align}
    \hat{\tau}^{\R} = \hat{\varphi}^{\R} - \hat{\varphi}^{\prime\R},
\end{align}
is a equivariant horizontal one form, hence the pull back
\begin{align} \label{redequone}
    \tau^{\R} =  j_2^{\R *} \hat{\tau}^{\R} = j_2^{\R *} \hat{\varphi}^{\R} - \tilde{\varphi}^{\R}
\end{align}
is also.

The induced connection $\tilde{\varphi}^{\R}$ is torsion free, since $\tilde{\theta}^{\R}$, the pull back of the solder form,
is again the solder form on $\Fr{SO}{\mu^{-1}(0),N}$. We pull back the structure equation $d\hat{\theta}^{\R} + \hat{\varphi}^{\R} \wedge \hat{\theta}^{\R} = 0$
to get
\begin{align}
    d\tilde{\theta}^{\R} + (j_2^{\R *}\hat{\varphi}^{\R}) \wedge \tilde{\theta}^{\R} = d\tilde{\theta}^{\R} + (\tilde{\varphi}^{\R} + \tau^{\R}) \wedge \tilde{\theta}^{\R} = 0.
\end{align}
Since $\tilde{\theta}^{\R}$ has values in $\R^{4n+k}$, we can split the equation
into the following two equations

\begin{align}
    d\tilde{\theta}^{\R} + \tilde{\varphi}^{\R} \wedge \tilde{\theta}^{\R} &= 0, \\
    \tau^{\R} \wedge \tilde{\theta}^{\R} &= 0, \label{symm}
\end{align}
which shows that $\tilde{\varphi}^{\R}$ is indeed torsion free.

$\tau^{\R}$ splits naturally into two forms with values in the top right matrices and bottom left matrices. Let $\tau_1^{\R}$ denote
the one that has values in the bottom left. Hence we have the splitting

\begin{align}
    j_2^{\R *} \hat{\varphi}^{\R} = \left(\begin{array}{cc} \phi_1^{\R} &
            -(\tau_1^{\R})^t \\ \tau_1^{\R} & \phi_2^{\R} \end{array}\right).
\end{align}

Using lemma (\ref{corr}) to identify $\tau_1^{\R}$ with a $(2,1)$-tensor on $\mu^{-1}(0)$, via
\begin{align}
    s(\tau_1^{\R})(\xi,\eta) = p\tau_1^{\R}(\bar{\xi})\tilde{\theta}^{\R}(\bar{\eta}),
\end{align}
where $p$ is a frame in $F(\mu^{-1}(0), M)$ and $\bar{\xi}, \bar{\eta}$ are lifts (compare lemma (\ref{corr})).

\begin{prop}[Second fundamental form]\label{secfunprop}
    $s(\tau_1^{\R})$ is the second fundamental form of $\mu^{-1}(0)$ in $M$.
\end{prop}
\begin{proof}
    In the next subsection we will show that $\tilde{\varphi}^{\R}$ is the pull back
    of the Levi Civita connection on $\Fr{SO}{\mu^{-1}(0)}$. The covariant
    derivative of a connection $\varphi$ with soldering form $\theta$ is given by
    \begin{align} \label{oneillcon}
        \nabla_{t}X = p(\bar{t}\theta(\bar{X}) + \varphi(\bar{t})\theta(\bar{X}_p)),
    \end{align}
    where $\bar{t}$ and $\bar{X}$ are lifts of the tangent vector $t$ and vector
    field $X$ to a frame $p$ (see e.g.~\cite[6.4]{bishop},
    but note that this book has a very unusual sign convention for the second fundamental form).
    Hence the second fundamental form is given by
    \begin{align}
        \SF(X, Y) &= \nabla_X^{M}Y - \nabla_X^{\mu^{-1}(0)}Y
        = p(j_2^{\R *} \hat{\varphi}^{\R}(\bar{X}) -
            \phi^{\R}_1(\bar{X}))\tilde{\theta}^{\R}(\bar{Y}) \\
        &= p(j_2^{\R *} \hat{\varphi}^{\R}(\bar{X}) -
        \tilde{\varphi}^{\R}(\bar{X}))\tilde{\theta}^{\R}(\bar{Y})
        =  p\tau^{\R}(\bar{X}_{p})\tilde{\theta}^{\R}(\bar{Y}_{p}).
    \end{align}
    Here we have used that $X$ and $Y$ are tangent to $\mu^{-1}(0)$ and hence $\phi_2^{\R}(\bar{X}_p)\tilde{\theta}^{\R}(\bar{Y}_p) = 0$.
    Note that $\SF$ is symmetric, because $\tau^{\R} \wedge \tilde{\theta}^{\R} = 0$, by equation (\ref{symm}). Since the second fundamental
    form is only defined for tangent vectors to $\mu^{-1}(0)$ and takes values orthogonal to $\mu^{-1}(0)$, we have to restrict
    $\tau^{\R}$ to $\tau_1^{\R}$ as described above.
\end{proof}

\begin{prop}[Second fundamental form as Hessian] Let $f\colon M \to V$ be a
    smooth map, where $M$ is a Riemannian manifold and $V$ a vector space.
    Assume further, that $0\in V$ is a regular value. $Df\colon TM \to V$
    identifies every fiber of the bundle $Tf^{-1}{(0)}^{\perp}$ with $V$, and
    under this identification the negative of the Hessian matrix of $f$ equals
    the second fundamental form of $f^{-1}(0)$ in $M$.  \end{prop}

\begin{proof}
    The first claim is just the dimension formula for a linear map,
    \begin{align}
        Df_p\colon Tf^{-1}(0)\oplus Tf^{-1}{(0)}^{\perp} \to V,
    \end{align}
    which has kernel $Tf^{-1}(0)$. Note that the second equality only holds for vector fields tangent to $f^{-1}(0)$, since
    the second fundamental form is only defined for these. Let $X$ and $Y$ be vector fields tangent to $f^{-1}(0)$. Then
    \begin{align}
        \text{Hess}(f)(X,Y) &= X(Yf) - Df(\nabla_X^{M}Y) \\
        &= X(\underbrace{Df(Y)}_{=0}) - \underbrace{Df(\nabla_X^{\mu^{-1}(0)}Y)}_{=0} - Df \SF(X,Y) \nonumber \\
        &= - Df\SF(X,Y) \nonumber
    \end{align}
\end{proof}
In this sense, $\tau_1^{\R}$ is associated with the $-\text{Hess}(\mu)$ by the
two aforementioned propositions.

\subsection{Forms on \texorpdfstring{$\Fr{SO}{\mu^{-1}(0)}$}{Fr(SO)(u(-1)(0))}}
Recall that the torsion free connection $\tilde{\varphi}^{\R}$ decomposes into two one forms $\phi_1^{\R}$ and $\phi_2^{\R}$.
$\phi_1^{\R}$ with values in $\mathfrak{so}(4n+k)$ induces a connection on $\Fr{SO}{\mu^{-1}(0)}$, because
\begin{align}
    \phi_1^{\R} ({(Dk_1^{\R})}^{-1}(0)) &= 0, \\
    R_g^*\phi_1^{\R} &= \phi_1^{\R} \qquad \forall g \in O(3k) \subset O(4m),
\end{align}
which is true because $Dk_1^{\R}\colon \mathfrak{so}(4n+k) \oplus
\mathfrak{so}(3k) \to \mathfrak{so}(4n+k)$ is the projection.
It allows us to define
\begin{align}
    \varphi^{\mu^{-1}(0)}(\eta) = \phi_1^{\R}(\tilde{\eta}), \qquad \tilde{\eta} \in {(Dk_1^{\R})}_q^{-1}(\eta),
\end{align}
i.e. $k_1^{\R *} \varphi^{\mu^{-1}(0)} = \phi_1^{\R}$.
Since the solder form on $\mu^{-1}(0)$ pulled back to $F(\mu^{-1}(0), M)$ is the form $\tilde{\theta}^{\R}$,
we get the equation
\begin{align}
    d\theta^{\mu^{-1}(0), \R} + \varphi^{\mu^{-1}(0)} \wedge \theta^{\mu^{-1}(0), \R} = 0,
\end{align}
and see that $\varphi^{\mu^{-1}(0)}$ is the unique Levi Civita connection on $\mu^{-1}(0)$.

\subsection{Riemannian Submersions} \label{rmsub} The next step involves
understanding Riemannian submersions on the level of frame bundles. Since there
is no exposition of this known to the author, we will describe it in a general
setting, and apply it to the reduction afterwards.

Let us at this point recall the basics of the Riemannian submersion theory of
O'Neill~\cite{oneill1966}.  A Riemannian submersion $\pi\colon M^m \to B^b$ is a
smooth map between two Riemannian manifolds such that $\pi$ is a submersion and
$D\pi_x|_{H_x} \colon H_{x} \to T_{\pi(x)}B$ is a isometry for all $x\in M$,
where $H_x$ is the orthogonal complement of $\ker(D\pi)\subset T_{x}M$.

To such a Riemannian submersion we may associate two important $(2,1)$-tensor
fields on $M$,

\begin{align} T_{X}Y &=
    \mathcal{H}\nabla^M_{\mathcal{V}X}\mathcal{V}Y +
    \mathcal{V}\nabla^M_{\mathcal{V}X}\mathcal{H}Y \\ A_{X}Y &=
    \mathcal{H}\nabla^M_{\mathcal{H}X}\mathcal{V}Y +
    \mathcal{V}\nabla^M_{\mathcal{H}X}\mathcal{H}Y,
\end{align}
where $\mathcal{H}$ and $\mathcal{V}$ are the horizontal and vertical projection
in $TM$, respectively. $T$ is known to be the second fundamental form of each
fiber (if vertical vector fields are plugged in), whereas $A$ is related to the
obstruction to integrability of the horizontal distribution on $M$. An important
fact is that \begin{align} A_{X}Y = \frac{1}{2}\mathcal{V}\left[X,Y\right],
\end{align} for horizontal vector fields $X$ and $Y$. If the Riemannian
submersion $\pi\colon M\to B$ should also happen to be a principal bundle, and
we fix the connection corresponding to the horizontal subspaces, then $2A_{X}Y =
-R(X,Y)$, where $R(X,Y)$ is the curvature of the connection, if we identify the
vertical tangent space with the Lie algebra as usual.

In the world of principal bundles this can be expressed the following way. Let
$\Fr{}{M}$ be the principal bundle of frames and $\Fr{}{B,M}$
the reduction to adapted frames on $M$. Here a frame is adapted if it respects the splitting of $TM$ into horizontal and vertical parts,
i.e.
\begin{align}
    \Fr{}{B,M} = \left\{p\in \Fr{}{M} : \text{$\text{im} (p|_{\R^{b}})$ is horizontal} \right\}.
\end{align}

Then a pull back of the Levi Civita connection $\phi$ on $\Fr{}{M}$ and the
solder form $\theta$ gives, after a suitable projection, a connection $\psi$ on
$\Fr{}{B,M}$ with structure equation
\begin{align} \label{oneillstruc}
    d\theta^{\prime} + \psi\wedge \theta^{\prime} + \tau\wedge \theta^{\prime} = 0,
\end{align}
where $\theta^{\prime}$ is the pull back of the solder form, $\psi$ the
projected connection and $\tau = i^*\phi - \psi$, where $i\colon \Fr{}{B,M} \to
\Fr{}{M}$
is the inclusion. We see that $\tau$ is an obstruction to the integrability of
the horizontal distribution, because for a product manifold $M = M_1\times M_2$
we have the commutative diagram
\begin{center}
    \begin{tikzpicture}
        \matrix (m) [matrix of math nodes,row sep=3em,column sep=4em,minimum width=2em]
        {
            & \Fr{}{M} & \\
            \Fr{}{M_1} & \Fr{}{M_1,M_2} & \Fr{}{M_2}\\
            M_1 & M & M_2 \\
        };
        \path[-stealth]
        (m-2-1) edge (m-3-1)
        (m-2-2) edge (m-1-2)
        (m-2-2) edge (m-2-1)
        (m-2-2) edge (m-2-3)
        (m-2-2) edge (m-3-2)
        (m-3-2) edge (m-3-1)
        (m-3-2) edge (m-3-3)
        (m-2-3) edge (m-3-3);
    \end{tikzpicture}
\end{center}
and the connection on $\Fr{}{M}$ reduces to a connection on $\Fr{}{M_1,M_2}$, which
is the sum of the connections pulled back from $\Fr{}{M_i}$.
On the other hand, from the construction of the last chapter, we also know that
$\tau$ is related to the second fundamental forms of the fibers.

The notion of horizontal and vertical projection extends to horizontal forms on
$\Fr{}{B,M}$, via
\begin{align}
    \tau_h(\xi) = \tau (\overline{\mathcal{H}D\pi^{\prime}(\xi)})  \\
    \tau_v(\xi) = \tau (\overline{\mathcal{V}D\pi^{\prime}(\xi)}),
\end{align}
where $\pi^{\prime}$ is the principal bundle map of $\Fr{}{B,M}$ and the over line is a lift with respect to that map. It is easy to see that this is well
defined for a horizontal form, since it does not depend on the choice of lift. Note also that by definition $\tau = \tau_h + \tau_v$. The following
proposition is the main result of this section.
\begin{prop}[O'Neill on Principal Bundles]\label{oneillprop}
    $\tau_v$ corresponds to $T$ and $\tau_h$ corresponds to $A$.
\end{prop}
\begin{proof}
    Note that $\tau$ is described by the difference of the connection on
    $\Fr{}{M}$ and the connection on $\Fr{}{B,M}$. The connection on $\Fr{}{M}$ gives rise to
    the covariant derivative $\nabla^M$, and the connection on $\Fr{}{B,M}$ to $\tilde{\nabla}$. As we have shown before, the connection extended from $\tilde{\nabla}$
    splits into two connections which are the Levi Civita connection on the fibers and the horizontal submanifolds, if they exist. Even if they
    do not, a quick inspection of equation (\ref{oneillcon}), using the matrix form of the reduced connection, shows that
    \begin{align}
        \tilde{\nabla}_{\xi}X = \mathcal{H}\nabla^M_{\xi}X
    \end{align}
    if $\xi$ and $X$ are horizontal and
    \begin{align}
        \tilde{\nabla}_{\eta}Y = \mathcal{V}\nabla^M_{\eta}Y,
    \end{align}
    if $\eta$ and $Y$ are vertical. The unique extension of this to $\Fr{}{M}$ gives the connection
    \begin{align}
        \hat{\nabla}_{\chi} Z := \mathcal{H}\nabla^M_{\chi}\mathcal{H} Z + \mathcal{V}\nabla^M_{\chi}\mathcal{V} Z,
    \end{align}
    for $\chi$ an arbitrary tangent vector and $Z$ an arbitrary vector field on $M$.
    This can be verified by showing that the above is indeed a covariant derivative on $M$ and that it restricts to $\tilde{\nabla}$ if both $\chi$ and $Z$
    are vertical, or both are horizontal. The latter is immediately clear, the former some simple calculations.

    We see now, that
    \begin{align}
        \nabla^M_{\chi}Z &= \mathcal{H}\nabla^M_{\mathcal{H}\chi}\mathcal{H}Z + \mathcal{H}\nabla^M_{\mathcal{V}\chi}\mathcal{H}Z + \mathcal{H}\nabla^M_{\mathcal{H}\chi}\mathcal{V}Z
        + \mathcal{H}\nabla^M_{\mathcal{V}\chi}\mathcal{V}Z \\
        &\qquad + \mathcal{V}\nabla^M_{\mathcal{H}\chi}\mathcal{H}Z + \mathcal{V}\nabla^M_{\mathcal{V}\chi}\mathcal{H}Z + \mathcal{V}\nabla^M_{\mathcal{H}\chi}\mathcal{V}Z
        + \mathcal{V}\nabla^M_{\mathcal{V}\chi}\mathcal{V}Z \nonumber \\
        &= A_{\chi}Z + T_{\chi}Z + \hat{\nabla}_{\chi}Z, \nonumber
    \end{align}
    hence the difference of connections indeed gives $A + T$. Finally, notice that if $\chi$ is horizontal then $T$ vanishes, as does $\tau_v$. If on the other hand
    $\chi$ is vertical, then $A$ vanishes, as does $\tau_h$.
\end{proof}

The principal bundle of frames $\Fr{}{B}$ of $B$ can be pulled back to $M$ via
$\pi$. The Levi Civita connection $\phi^B$ on $\Fr{}{B}$ can also be pulled back to a connection
$\tilde{\phi}$ on $\pi^*\Fr{}{B}$ together with the structure equation
\begin{align}
    \tilde{\phi} + \tilde{\theta}_B \wedge\tilde{\phi} = 0,
\end{align}
where $\tilde{\theta}_B$ is the pull back of the solder form $\theta_B$ on
$\Fr{}{B}$. If we pull this solder form into $\Fr{}{B,M}$, we get a form $\theta_B^{\prime}$, where
the obvious restriction map is used $k\colon \Fr{}{B,M} \to \pi^*\Fr{}{B}$. A
calculation similar to that in remark (\ref{naturalsoldering}) shows that $\theta_B^{\prime}$ agrees with the part
of $\theta^{\prime}$, that has values in $\R^{b}$. If we split $\theta^{\prime}$ into two parts, $\theta_1$ and $\theta_2$ with values in $\R^{b}$
and $\R^{m - b}$, and $\psi$ into $\psi_1$ and $\psi_2$ with values in
$\mathfrak{so}(b)$ and $\mathfrak{so}(m-b)$,
then the structural equation (\ref{oneillstruc}) of $\psi$ decomposes into
\begin{align}
    d\theta_1 + \psi_1\wedge\theta_1 + \tau\wedge \theta_2 &= 0 \\
    d\theta_2 + \psi_2\wedge\theta_2 + \tau\wedge\theta_1 &= 0.
\end{align}
If we restrict the first equation to $\pi$-horizontal vectors, the last term vanishes and we see that $\psi_1$ is the Levi Civita connection pulled back from $B$.
Such a restriction also turns $\tau$ into $\tau_h$ and we get the formula
\begin{align}
    k^*\pi^*\phi^B + \tau_h = i^*\phi_M,
\end{align}
on $\Fr{}{B,M}$, if we restrict to vectors lifted from $B$. This is the recovery of
O'Neill's formula for the connections~\cite[Lemma 3.4]{oneill1966}.

\subsection{Forms on \texorpdfstring{$\Fr{SO}{N,\mu^{-1}(0)}$}{Fr(SO)(N,u(-1)(0))}}
Applying the last section to the reduction $\Fr{SO}{N,\mu^{-1}(0)}$ of $\Fr{SO}{\mu^{-1}(0)}$ on $\mu^{-1}(0)$, we get the equation
\begin{align} \label{redequtwo}
    j_3^{\R *} \varphi^{\mu^{-1}(0)} = \psi_1 + \psi_2 + \tau^{\prime},
\end{align}
where $\psi_1$ is the pull back of the Levi Civita connection on $N$.

\subsection{Forms on \texorpdfstring{$\Fr{Sp}{N,M}$}{Fr(Sp(N,M))}}
Now we will do a similar construction on the quaternionic side of the reduction for $\Fr{Sp}{N,M}$
As with $\Fr{SO}{\mu^{-1}(0), M}$, $\Fr{Sp}{N,M}$ will in general not be horizontal in $\iota^* \Fr{Sp}{M}$. Using Proposition~\ref{reduction},
we construct a connection $\tilde{\varphi}^{\Qu}$ with the decomposition
\begin{align} \label{qusplit}
    \mathfrak{sp}(m) = \mathfrak{sp}(n) \oplus \mathfrak{o}(k) \oplus \mathfrak{f},
\end{align}
induced by an inclusion of $\Sp(n)\times \SO(k)$ in $\Sp(m)$ as described in the beginning. As before, the obvious choice of
complement will satisfy the necessary condition (\ref{reductionCond}).

We get the projected connection form $\tilde{\varphi}^{\Qu}$ which decomposes into two equivariant one-forms $\phi_1^{\Qu}$ and $\phi_2^{\Qu}$ with
values in $\mathfrak{sp}(n)$ and $\mathfrak{so}(k)$ respectively and a difference form $\tau^{\Qu}$ with
\begin{align} \label{quickqu}
    \phi_1^{\Qu} + \phi_2^{\Qu} + \tau^{\Qu} = j_2^{\Qu *} \hat{\varphi}^{\Qu}.
\end{align}

\section{Final Result}\label{sec-final}

\subsection{Preparation}
Let us recall the connections of the real reductions. On $\Fr{SO}{\mu^{-1}(0), M}$ we have equation (\ref{redequone})
\begin{align} \label{quickref}
    \phi_1^{\R} + \phi_2^{\R} + \tau^{\R} = j_2^{\R *} \hat{\varphi}^{\R},
\end{align}
where $\hat{\varphi}^{\R}$ is the pull back of the Levi Civita connection on $M$. $\phi_1^{\R}$ is the pull back of the Levi-Civita connection of $\mu^{-1}(0)$,
which in turn decomposes on $\Fr{SO}{N,\mu^{-1}(0)}$ according to equation (\ref{redequtwo}).

The connection $\psi_1 + \psi_2$ on $\Fr{SO}{N,\mu^{-1}(0)}$ can be extended back to a connection $\tilde{\psi}_1 + \tilde{\psi}_2$ on $\Fr{SO}{\mu^{-1}(0)}$, so that we have
\begin{align}
    \tilde{\psi}_1 + \tilde{\psi}_2 + \tilde{\tau}^{\prime} = \varphi^{\mu^{-1}(0)},
\end{align}
where $\tilde{\tau}^{\prime}$ is defined by this equation (and hence the pull
back of it is $\tau^{\prime}$.) So if we pull back this equation to
$\Fr{SO}{\mu^{-1}(0), M}$, we get
\begin{align}
    k_1^{\R *} \tilde{\psi}_1 + k_1^{\R *} \tilde{\psi}_2 + k_1^{\R *} \tilde{\tau}^{\prime} = \phi_1^{\R},
\end{align}
and combining this with (\ref{quickref})

\begin{align}\label{realfinaleq}
    k_1^{\R *} \tilde{\psi}_1 + k_1^{\R *} \tilde{\psi}_2 + k_1^{\R *} \tilde{\tau}^{\prime} + \phi_2^{\R} + \tau^{\R} = j_2^{\R *} \hat{\varphi}^{\R}.
\end{align}

Since $i_3^{*}\hat{\varphi}^{\R} = \hat{\varphi}^{\Qu}$, we can identify the
right hand side of the equation above and of (\ref{quickqu}) if we pull back by
$i_3$,

\begin{align}\label{finalequ}
    i_3^*\left(k_1^{\R *} \tilde{\psi}_1 + k_1^{\R *} \tilde{\psi}_2 + k_1^{\R *} \tilde{\tau}^{\prime} + \phi_2^{\R} + \tau^{\R}\right) = \phi_1^{\Qu} + \phi_2^{\Qu} + \tau^{\Qu}.
\end{align}

To understand which terms correspond, it is a good idea to visualize where the
different forms take their values. If we identify $\Qu^n$ with $\R^{4n}$ such
that $a + ib + jc + kd$ gets mapped to $(a, b, c, d)$ ($a,b,c,d \in \R^n$), we
identify $n\times n$ quaternionic matrices $A + i B + j C + kD$ with $4n\times
4n$ real matrices of the form

\begin{align}\label{prototype}
    \left(\begin{array}{cccc} A & - B & -C & -D \\ B & A & -D & C \\ C & D & A & -B \\ D & -C & B & A\end{array}\right).
\end{align}

If we use a frame $p\in \Fr{Sp}{N,M}$ to identify $\iota^*(TM)$ with $\R^{4m}$, we see that both sides of the equations take values in matrices of the form
\begin{align} \label{matequ}
    \left(\begin{array}{cc} M_1 & -M_2^t \\ M_2 & M_3 \end{array}\right),
\end{align}
where $M_1$ is a $4n\times 4n$, $M_2$ a $4k\times 4n$ and $M_3$ a $4k\times 4k$
block matrix of the type given above. Using the quaternionic splitting, we can
decompose the $M_i$ into $A_i, B_i, C_i$ and $D_i$. Note that in $M_3$ only
$A_3$ (the diagonal) is non vanishing, because of the inclusion $SO(k)
\hookrightarrow Sp(k)$, $A \mapsto A + iA + j A + k A$.

The components of the matrices $M_i$ are of course only defined up to the choice
of frame $p\in \Fr{Sp}{N,M}$. However, two different frames differ by a matrix
in $Sp(n)\times SO(k)$, which leaves the components of $M_3$ and the
component-rows of $M_2$ invariant. $M_1$ and the columns of $M_2$ get
transformed by conjugation with a $\Qu$-linear matrix.

Define the matrix $M_2^1$ to be the first $k$ rows of $M_2$, $M_2^2$ to be
the other $3k$ rows and $M_3^1$ as $A_3$, ${(M_3^2)}^t = (B_3, C_3, D_3)$ and
$M_3^3$ as the matrix $M_3$ without the first $k$ columns and first $k$ rows.
Hence we may write (\ref{matequ}) as
\begin{align*}
    \left(\begin{array}{ccc}\begin{array}{ccc} & &  \\ & M_1 & \\ & & \end{array} & -{(M_2^1)}^t & -{(M_2^2)}^t \\ M_2^1 & M_3^1 & -{(M_3^2)}^t \\
            M_2^2 & M_3^2 & M_3^3 \end{array}\right).
\end{align*}

Starting with the right hand side of the equation (\ref{finalequ}),
$\phi_1^{\Qu}$ takes values $M_1$, $\phi_2^{\Qu}$ in $M_3$ and $\tau^{\Qu}$ the
remaining $M_2$ matrix. On the left hand side, $\tilde{\psi}_1$ takes values in
the $M_1$, $\tilde{\psi}_2$ in $M_3^1$, $\tilde{\tau}^{\prime}$ in $M_2^1$,
$\phi_2^{\R}$ in $M_3^3$ and $\tau^{\R}$ in the remaining $M_2^2$ and $M_3^2$
matrices.

\subsection{The Results}

The equations (\ref{finalequ}) and the following analysis of the previous section allows us to recover some of the results from~\cite{gocho1992}. First we see that
\begin{align} \label{finalone}
    i_3^* k_1^{\R *} \tilde{\psi}_1 = \phi_1^{\Qu} \qquad \Rightarrow \qquad k^*\psi_1 = \phi_1^{\Qu},
\end{align}
because both sides take values in $M_1$. If we pull back the Levi-Civita connection on $\Fr{SO}{N}$ to $\Fr{Sp}{N,M}$ via $\Fr{Sp}{N}$, we get $\phi_1^{\Qu}$ because of this
equation. Hence the pull back to $\Fr{Sp}{N}$ takes values in $\Qu$-linear
matrices, in other words the connection reduces to one on $\Fr{Sp}{N}$. This
shows that $N$ is indeed a hyperkähler manifold.

A more constructive argument can be given by noting that the Levi-Civita connection on $M$ is $G$-invariant, for the canonical choice of extension of the $G$ action
to $\Fr{SO}{M}$. This remains true for $\phi_1^{\Qu}$ and a careful examination shows that it can be pushed down to $\Fr{Sp}{N}$.

If we continue with $M_3$, we see that for $\xi \in \mathfrak{g}$, $\SF(\cdot, \xi)$, which is described by $M_3^2 = 0$, vanishes.

The fact that $M_3$ is only non-vanishing on the diagonal, gives a connection
between the covariant derivative on the fibers of $\pi \colon \mu^{-1}(0) \to
N$, and the normal derivative of $\mu^{-1}(0)$ described by $\phi_2^{\R}$, i.e.
$D_{\xi} Y := \text{pr}_{T\mu^{-1}{(0)}^{\perp}}\nabla^M_{\xi} Y$, for $\xi \in
T\mu^{-1}(0)$ and $Y\in \Gamma(\mu^{-1}(0), T\mu^{-1}{(0)}^{\perp})$ (see
e.g.~\cite[VII]{nomizu}). Precisely, we have for all $A\in\left\{I,J, K\right\}$
\begin{align}
    \nabla^{\text{F}}_{\xi} X = d\mu^A \circ D_{\xi}(AK^{\eta}), \qquad \forall \xi,\eta \in \mathfrak{g},
\end{align}
where $\nabla^{\text{F}}$ is the connection on the fiber.

Let us now focus on $M_2$. From proposition $(\ref{secfunprop})$ we know that $M_2^2$ and $M_3^2$ give the second fundamental form and from proposition
$(\ref{oneillprop})$ we know that $M_2^1$ is $A + T$, the O'Neill tensors. Hence
\begin{align}
    M_2(\xi) = p^{-1}\circ \left(\begin{array}{cccc} (A_{\xi}+T_{\xi})(\cdot) & (A_{\xi}+T_{\xi})(I\cdot) &  (A_{\xi}+T_{\xi})(J\cdot) & (A_{\xi}+T_{\xi})(K\cdot) \\
            \SF^{I}(\xi, \cdot ) & \SF^{I}(\xi, I \cdot )  & \SF^{I}(\xi, J\cdot ) & \SF^{I}(\xi, K\cdot ) \\
            \SF^{J}(\xi, \cdot ) & \SF^{J}(\xi, I \cdot )  & \SF^{J}(\xi, J\cdot ) & \SF^{J}(\xi, K\cdot ) \\
            \SF^{K}(\xi, \cdot ) & \SF^{K}(\xi, I \cdot )  & \SF^{K}(\xi, J\cdot ) & \SF^{K}(\xi, K\cdot )
        \end{array}\right)\circ p,
\end{align}
where $\SF^{A}$ is the second fundamental form of $\mu^{-1}(0) \hookrightarrow
M$ projected onto $A\mathfrak{g} \subset T\mu^{-1}{(0)}^{\perp}$ and $p\in \Fr{Sp}{N,M}$ is a frame (restricted in a suitable way).
Using the form (\ref{prototype}) of the matrix, we get the following results (recall the notation $\iota^*(TM) = H \oplus \mathfrak{g}\otimes_{\R}\Qu$).

If $\xi \in H$ and $\cdot \in H$, then the first row of $M_2$ becomes $-\frac{1}{2}R(\xi, \cdot), \ldots$,
where $R$ is the curvature of $\mu^{-1}(0) \to N$ as discussed before. This yields that for all $\xi,\eta \in H$,
\begin{align}
    -\frac{1}{2}R(\xi,\eta) = \SF^I(\xi, I\eta) = \SF^J(\xi, J\eta) = \SF^K(\xi, K\eta).
\end{align}
Here $\SF^I = d\mu^{I}\circ \SF$.
Note that this in particular implies that $R$ is hyperholomorphic, i.e.\ of type $(1,1)$ with respect to all complex structures (on $N$, viewing $R$ as a two form on $N$).

If $\xi \in \mathfrak{g}$ and $\cdot \in H$, then the first row becomes $T_{\xi}\cdot = \mathcal{V} \nabla_{\xi}^{\mu^{-1}(0)}\cdot, \ldots$, where $\nabla^{\mu^{-1}(0)}$
is the Levi-Civita connection on $\mu^{-1}(0)$ and $\mathcal{V}$ is the vertical projection in $T\mu^{-1}(0)$ from $\pi\colon \mu^{-1}(0) \to N$. This can be described
as the negative of the Weingarten map $\mathcal{W}_\xi(\cdot)$ of the fibers of $\pi$. Hence we get for all $\xi\in\mathfrak{g}, \eta \in H$,
\begin{align*}
    -\mathcal{W}_\xi(\eta) = \SF^I(\xi, I\eta) = \SF^J(\xi, J\eta) = \SF^K(\xi, K\eta).
\end{align*}
However, since $\SF$ is symmetric, $\SF(\xi,\cdot) = 0$, hence the Weingarten map
of the fibers vanish, in other words, the fibers are totally geodesic.

If $\xi \in H$ and $\cdot \in \mathfrak{g}$, the discussion needs to be carried out in $-M_2^t$. Using the formula for $A$ and $T$ (and that $\SF(\xi,\cdot) = 0$),
we see that
\begin{align}
    \text{pr}_{H}\circ \nabla^{\mu^{-1}(0)}_{\xi} X = 0,
\end{align}
for all $\xi \in H$ and $X\in \Gamma(\mu^{-1}(0),\mathfrak{g})$, which is already clear from $\mathcal{W}_{X}(\xi) = 0$. Both $\xi$ and $\cdot$ in $\mathfrak{g}$
again yield that the second fundamental forms of the fibers of $\pi$ vanish.

\bibliographystyle{plain}
\bibliography{reduction}
\vspace{1cm}
{\setlength{\parindent}{0cm}
\textsc{
Robin Raymond\\
Mathematisches Intitut\\
Georg-August-Universität Göttingen\\
Bunsenstraße 3-5\\
37073 Göttingen, Germany\\
}
\textsc{Email: }\href{mailto:robin.raymond@mathematik.uni-goettingen.de}{\url{robin.raymond@mathematik.uni-goettingen.de}}
}

\end{document}